\documentclass{article}[12pt]

\usepackage{amsmath, amsthm, amssymb, amsfonts, mathtools, bbm}
\usepackage{enumerate}
\usepackage{mathrsfs}

\usepackage{bbm}

\DeclareFontFamily{U}{mathx}{}
\DeclareFontShape{U}{mathx}{m}{n}{<-> mathx10}{}
\DeclareSymbolFont{mathx}{U}{mathx}{m}{n}
\DeclareMathAccent{\widehat}{0}{mathx}{"70}
\DeclareMathAccent{\widecheck}{0}{mathx}{"71}

\newcommand{\norm}[1]{\lVert#1\rVert}

\newtheorem{theorem}{Theorem}
\newtheorem{lemma}{Lemma}

\newtheorem{definition}{Definition}

\theoremstyle{remark}
\newtheorem{remark}{Remark}
              
\author{Alex Bergman}
\title{Inner approximations of doubling weights with applications to Beurling-Malliavin theory in Toeplitz kernels}
\date{\today}

\begin{document}

\maketitle

\begin{abstract}
    A meromorphic inner function is a bounded holomorphic function in the upper half-plane which is unimodular on the real line and extends to a meromorphic function in the whole complex plane. The argument of a meromorphic inner function on the real line is a strictly increasing function. It turns out that it is important for many problems in function theory to approximate an arbitrary increasing function, $f$, by the argument of a meromorphic inner function. Depending on desired approximation this is a delicate problem. In this paper consider the case when $f$ satisfies a doubling condition. We give two applications of our main result. The first is a sufficient density condition for a set $\Lambda$ to be a zero set for a Toeplitz kernel with real analytic and unimodular symbol. Our second application is to describe a class of admissible Beurling-Malliavin majorants in model spaces. The generality considered here lets us treat most cases of model spaces generated by meromorphic one-component inner functions.
\end{abstract}

\section{Introduction}

An inner function is a holomorphic map $\Theta : \mathbb{C}_{+} \to \mathbb{D}$ which is maximal in the sense that
\begin{equation*}
    \lim_{y\to 0} \lvert \Theta(x+iy)\rvert = 1, \text{ for almost every } x \in \mathbb{R}.
\end{equation*}
Inner functions play a central role in the theory of Banach and Hilbert spaces of analytic functions and, in particular, with its connections to harmonic analysis and spectral theory in one variable. Let us give a brief sketch of how they arise in the theory of Schrödinger operators. Let $V \in L^{1}_{loc}(0,b)$, where $0<b\leq \infty$ be real-valued, and consider the eigenvalue problem
\begin{equation}\label{eq:schrödinger}
    -\partial_{x}^{2}u(x,z) + V(x)u(x,z) = zu(x,z), \text{ for all } x\in (0,b), \text{ and }z \in \mathbb{C}.
\end{equation}
We suppose that the left end-point is regular (i.e., $V$ is integrable at $0$), but we allow the right end-point to be singular. We fix a separated self-adjoint boundary condition at $b$. For example in the limit point case this means that we assume our solution belongs to $L^{2}$ and if $b < \infty$ and $V$ is integrable on all $(0,b)$ we may choose, for example, $u(b,z) = 0$ (Dirichlet) or $u'(b,z) = 0$ (Neumann). The Weyl $m$-function is then
\begin{equation*}
    m(z) = \frac{u'(0,z)}{u(0,z)}, \text{ for all } z \in \mathbb{C} \setminus \mathbb{R},
\end{equation*}
where the derivative is in $x$. It is well-known that $m$ maps the upper half-plane into itself and hence
\begin{equation*}
    \Theta(z) = \frac{m(z)-i}{m(z)+i},
\end{equation*}
maps into the unit disc. It turns out that $\Theta$ is an inner function if and only if the spectrum is purely singular. Moreover, the solution
\begin{equation*}
    \omega(x,z) = \frac{u(x,z)}{u'(0,z)+iu(0,z)},    
\end{equation*}
does not depend on the choice of solution $u$ and the map
\begin{equation*}
    Uf(z) = \frac{1}{\sqrt{\pi}}\int_{0}^{b} f(x)\omega(x,z)dx, \text{ for all } z \in \mathbb{C}_{+},    
\end{equation*}
is a unitary map from $L^{2}(0,b)$ onto the model space $K_{\Theta} = H^{2} \ominus \Theta H^{2}$ (to be defined shortly). Indeed, the reproducing kernel for the space $\left\{ Uf : f \in L^{2}(0,b)\right\}$ is given by the formula
\begin{equation*}
    K(\lambda,z) = \frac{1}{\pi}\int_{0}^{b} \omega_{\overline{\lambda}}(x)\omega_{z}(x)dx = \frac{1-\overline{\Theta(\lambda)}\Theta(z)}{2\pi i(z-\overline{\lambda})}.
\end{equation*}
Thus this procedure transforms basis problems and spectral problems for Schrödinger equations to problems in complex function theory. The most well-known example of this approach is the work of Khrushchev, Nikolski, and Pavlov on Pavlov's solution to the Riesz basis problem for exponentials, see the classical paper \cite{MR643384}. This method has recently found great success in the inverse spectral theory, see for example \cite{MR2739783, MR1943095, MR4033521}. Let us also mention that this program is realizable for more general equations, such as canonical systems, see for example \cite{MR4639943, aleman2025invariantsubspacesgeneralizeddifferentiation, MR3890099}.\newline

In this article we shall be concerned only with those inner functions which extend to meromorphic functions in the whole complex plane. In the language of spectral theory this corresponds to the assumption of the resolvent of the self-adjoint operator associated with equation \eqref{eq:schrödinger} being compact. Meromorphic inner functions are characterized by their Nevanlinna factorization being of the form
\begin{equation*}
    \Theta(z) = B(z)S_{a}(z),    
\end{equation*}
where $B$ is a Blaschke product whose zeros $\left\{ z_{n} \right\}$ accumulate only at infinity and $S_{a}(z) = \exp(iaz)$, for some $a \geq 0$. By multiplying by a unimodular factor we may assume, without loss of generality, that $\Theta(0) = 1$. From now on we make this assumption. Since $\Theta$ is unimodular on the real-line we may write
\begin{equation*}
    \Theta(x) = \exp(i\arg\Theta(x)).
\end{equation*}
Note that the argument of a meromorphic inner function on the real line can be chosen to be a real-analytic. It is unique if we assume that it is continuous and fix the value at the origin to be zero (which we do from now on). Its derivative can be expressed in terms of the Nevanlinna factorization as follows
\begin{equation*}
    \arg \Theta'(x) = a + 2\sum_{n} \frac{y_{n}}{(x-x_{n})^{2} + y_{n}^{2}}, \text{ for all } z_{n} = x_{n}+iy_{n}.
\end{equation*}
This also shows that the argument is an increasing function. Central problems in the theory of analytic functions spaces include the distribution of zeros. A typical problem encountered in the theory is that of approximating a smooth and strictly increasing map $f : \mathbb{R} \to \mathbb{R}$ by the argument of a meromorphic inner function, $\Theta$, see the articles \cite{MR3631454, alexei_rupam_bounded_phase} which are dedicated to this problem and the articles \cite{MR2215727, MR2609247, MR2910798, MR2628803} where such an approximation plays a direct and decisive role. If one settles for the requirement that $f-\arg \Theta$ be bounded this is not a difficult problem. However, typically one is also required to control the derivative $\arg \Theta'(x)$ for real $x$ by a quantity related to the derivative of $f$. Our main result allows one to do this in the case that $f'$ behaves like $\alpha'$ for some $\alpha$ which is the argument of a meromorphic inner function and such that the measure
\begin{equation*}
    \mu_{\alpha}(I) = \int_{I} \alpha'(x)dx,
\end{equation*}
is a local doubling measure in the following sense: for an interval $I \subset \mathbb{R}$ we denote by $2I$ the interval with twice the length of $I$ and the same center. A Borel measure $\mu$ is called doubling if there exists a constant $C > 0$, such that for all bounded intervals $I$ the inequality $\mu(2I) \leq C\mu(I)$ is satisfied. It is called locally doubling if $\mu(2I) \leq C\mu(I)$, whenever $\mu(2I) \leq 1$. In fact, we shall also obtain more general results when the measure (not necessarily analytic) satisfies a local doubling condition (to be stated shortly).\newline

Since it may be quite difficult to see the applicability of such a result for the non-expert we include two applications of our result. The first is a general criterion for a set $\Lambda$ to be a zero set in a Toeplitz kernel generated by a unimodular and real-analytic symbol. This includes model and de Branges spaces. The second concerns the construction of a model space function $f$ which is majorized on the real-line by a given weight function $\omega$. This is in the theme of the Beurling-Malliavin majorant theorem which looks for weights which have a non-trivial minorant with Fourier spectrum contained in a fixed interval. We begin with some terminology.\newline

Let $H^{p}$ denote the Hardy space of the upper half-plane, $\mathbb{C}_{+}$, consisting of all analytic functions subject to the growth condition
\begin{equation*}
    \norm{f}_{p}^{p} = \sup_{y > 0} \int_{-\infty}^{\infty} \lvert f(x+iy) \rvert^{p}dx < \infty, \text{ for all } 1 \leq p < \infty.
\end{equation*}
For $p = \infty$ we let $H^{\infty}$ denote the set of bounded analytic function in the upper half-plane. Standard references for Hardy spaces include \cite{MR2261424} and \cite{MR1669574}.Functions in $H^{p}$ have non-tangential boundary limits Lebesgue almost-everywhere on the real line. We freely identify an $H^{p}$ function with its extension to the real-line. Moreover, the analytic function can be recovered from its boundary values by the means of the Poisson integral formula.\newline

If $\alpha : \mathbb{R} \to \mathbb{R}$ is strictly increasing we define the metric
\begin{equation*}
    d_{\alpha}(x,y) = \lvert \alpha(x)-\alpha(y) \rvert.
\end{equation*}
This will be the natural metric for our problems. If $\alpha : \mathbb{R} \to \mathbb{R}$ is the argument of a meromorphic inner function then the measure $d\mu_{\alpha}(x) = \alpha'(x)dx$ is locally doubling if and only if one (and hence both) of the following conditions hold
\begin{enumerate}[(i)]
    \item There exists constants $M>0$ and $m > 0$, such that $m\alpha'(y) \leq \alpha'(x) \leq M \alpha'(y)$, for all $x, y \in \mathbb{R}$ such that $d_{\alpha}(x,y) \leq 1$.
    \item There exists a constant $C > 0$, such that $\lvert \alpha''(x) \rvert \leq C\alpha'(x)^{2}$.
\end{enumerate}
For a proof we refer to \cite{MR2944102}. Let us also remark that Aleksandrov has shown that property $(ii)$ characterizes the so-called \textit{one-component} inner functions. More on this in the section on preliminaries. We take these properties to be the definition of a \textit{regular locally doubling weight}.

\begin{definition}
    Let $\alpha : \mathbb{R} \to \mathbb{R}$ be a smooth strictly increasing map. We say that $\alpha$ is a \textit{regular locally doubling weight} if
    \begin{enumerate}[(i)]
    \item there exists constants $M,m > 0$ and $R > 0$, such that $m\alpha'(y) \leq \alpha'(x) \leq M \alpha'(y)$, for each $x, y \in \mathbb{R}$ such that $d_{\alpha}(x,y) \leq 1$ and $\lvert x \rvert, \lvert y \rvert \geq R$.
    \item There exists a constant $C > 0$ and $R > 0$, such that $\lvert \alpha''(x) \rvert \leq C\alpha'(x)^{2}$, for $\lvert x \rvert \geq R$.
\end{enumerate}
\end{definition}

Note that we only require the properties to hold away from the origin. The simplest examples of locally doubling weights are the power functions $\lvert x \rvert^{\kappa}$ with $\kappa > 0$. We are now ready to state our main theorem. We use the notation $\langle x \rangle = (1+x^{2})^{1/2}$ and $f \lesssim g$ for two positive functions $f$ and $g$ means that there exists a constant $C > 0$, such that $f \leq Cg$.

    \begin{theorem}\label{thm:main}
        Let $\alpha : \mathbb{R} \to \mathbb{R}$ be a smooth locally doubling weight satisfying
        \begin{equation*}
                \langle x \rangle^{-\kappa} \lesssim \alpha'(x), \text{ for some } \kappa < 1.
        \end{equation*}
        Then there exists a meromorphic inner function $J$ and a positive integer $N_{0}$, such that
        \begin{equation*}
            \lvert \alpha(x) -\arg J(x) \rvert \leq 2\pi, \text{ for all } x \in \mathbb{R},
        \end{equation*}
        and
        \begin{equation*}
            \lvert J'(x) \rvert \lesssim \alpha'(x)^{N_{0}}\langle x \rangle^{N_{0}}, \text{ for all }x \in \mathbb{R},
        \end{equation*}
        and
        \begin{equation*}
            \langle \lambda \rangle^{-N_{0}}\alpha'(\lambda)^{-N_{0}} \lesssim \lvert J'(\lambda) \rvert, \text{ for all } \lambda \in \Lambda,
        \end{equation*}
        where $\Lambda = \left\{ \lambda \in \mathbb{R} : \alpha(\lambda) = 0 \mod 2\pi\right\}$.
    \end{theorem}

\subsection{Applications}

Since the proof of the main construction will be quite technical we wish to describe our applications first. The two applications and their corresponding sections can be read independently of each other.

\subsection*{Zero sets in Toeplitz kernels}

Let $P_{+}$ denote the orthogonal projection from $L^{2}$ onto $H^{2}$ and $U \in L^{\infty}(\mathbb{R})$. The Toeplitz operator with symbol $U$ is the map
\begin{equation*}
    T_{U} : H^{2} \to H^{2}, \; \; T_{U}f = P_{+}(Uf).
\end{equation*}
Toeplitz operators are ubiquitous in harmonic analysis, complex analysis, and operator theory. In this article we shall be concerned with their kernels. The kernel of the operator $T_{U}$ is given by
\begin{equation*}
    N^{2}(U) = \left\{ f \in H^{2}: \overline{Uf} \in H^{2}\right\}.   
\end{equation*}
Toeplitz kernels are important in part because they include the so-called model subspaces and closely related de Branges spaces which generalize the classical Paley-Wiener space. Because of this connection and the connection between such spaces and spectral theory, as outlined above, Toeplitz kernels give a unified approach to study completeness problems of many systems of special functions in one variable, see for example \cite{MR643384, MR2215727}.

In this article we shall study the more general $L^{p}$ Toeplitz kernels. We denote by $N^{+}$ the Smirnov class in the upper half-plane
\begin{equation*}
\begin{split}
    N^{+} = \left\{ f \in \mathcal{O}(\mathbb{C}_{+}) : f = h/g \text{, with } h,g \in H^{\infty} \text{, and } g \text{ outer}\right\}, \\
    \text{ where } \mathcal{O}(\mathbb{C}_{+}) \text{ is the set of analytic functions in } \mathbb{C}_{+}.
\end{split}
\end{equation*}
In analogy with the case $p=2$ we define the $p$-Toeplitz kernel of $U$ as
\begin{equation*}
    N^{p}(U) = \left\{ f\in H^{p} \cap L^{1}_{loc} : \overline{Uf} \in H^{p} \right\},
\end{equation*}
and the Smirnov Toeplitz kernel
\begin{equation*}
    N^{+}(U) = \left\{ f \in N^{+} \cap L^{1}_{loc} : \overline{Uf} \in N^{+} \right\}.
\end{equation*}
Note that the elements in $N^{+}(U)$ are assumed to be locally summable. This is is a natural condition when $U$ has only one singularity (the one at infinity) and will make the elements real analytic if $U$ is.\newline

The problem which we consider is the problem of describing the discrete sets $\Lambda \subset \mathbb{R}$ which are zero sets in $N^{\infty}(U)$. In \cite{MR2609247} Makarov and Poltoratski gave an answer to this problem (up to an $\epsilon$ gap) for symbols of the form $U = \overline{\Theta}$ with $\Theta$ a meromorphic inner function satisfying $\arg \Theta' \sim \lvert x \rvert^{\kappa}$, for some $\kappa \geq 0$. Their conditions are in terms of triviality (and non triviality) of certain Toeplitz kernels. Here we shall give a geometric condition in terms of an analogue of the usual Beurling density. Our result also allows for more general locally doubling weights. Although, the results are not as strong as in \cite{MR2609247} they also require less work and may indicate that the results in \cite{MR2609247} extend to a larger class of locally doubling weights.\newline

We shall restrict our attention to unimodular and real analytic symbols of the form
\begin{equation*}
    U = \exp(i\arg U), \; \; \arg U(x) = -\alpha(x) + h(x),
\end{equation*}
with $\alpha$ a real analytic locally doubling weight such that there exists $N, C, c > 0$, such that
\begin{equation*}
    c <\alpha'(x) < C\langle x \rangle^{N},
\end{equation*}
and $h$ is bounded and Lipschitz continuous, or more generally is bounded and satisfies an estimate of the form
\begin{equation*}
    \lvert h'(x) \rvert \leq C\langle x \rangle^{K}, \text{ for some } K \geq 0.
\end{equation*}
In this case all elements of $N^{p}(U)$ will be real analytic functions. This motivates the discreteness condition on $\Lambda$.\newline

To formulate our main result we introduce the $\alpha$ upper density for any smooth and strictly increasing function $\alpha$ and $\Lambda \subset \mathbb{R}$. For such $\alpha$ and $\Lambda$ we define
\begin{equation*}
    D_{\alpha}^{+}(\Lambda) = \lim_{r \to \infty}\sup_{\alpha(I) = r} \frac{\#(\Lambda \cap I)}{r}.
\end{equation*}
If $\alpha = x$ we obtain the usual Beurling upper density. We say that a set is $\alpha$-uniformly separated if there exists $\delta > 0$, such that
    \begin{equation*}
        d_{\alpha}(\lambda,\lambda') > \delta, \text{ for all } \lambda, \lambda' \in \Lambda \text{ with } \lambda \neq \lambda'.
    \end{equation*}

\begin{theorem}\label{thm:zero_function}
            Let $\arg U = -\alpha + h$, $\alpha, h$ real-analytic. Suppose $\alpha$ is a regular locally doubling measure satisfying
        \begin{equation*}
            1\lesssim \alpha'(x) \lesssim \langle x \rangle^{K_{0}}, \text{ for some } K_{0} \geq 0.
        \end{equation*}
        and $h \in L^{\infty}$ satisfies
        \begin{equation*}
            \lvert h'(x) \rvert \lesssim \langle x \rangle^{K_{1}}, \text{ for some } K_{1} \geq 0.
        \end{equation*}
        Let $\Lambda$ be a $\alpha$-uniformly separated subset of the line, such that $D_{\alpha}^{+}(\Lambda)< 1/2\pi$. Then there exists $T \in N^{\infty}(U)$, such that
        \begin{equation*}
            T(\lambda) = 0, \text{ for all } \lambda \in \Lambda,
        \end{equation*}
        and
        \begin{equation*}
            \langle \lambda\rangle^{-N} \lesssim \lvert T'(\lambda) \rvert , \text{ for all  } \lambda \in \Lambda.
        \end{equation*}
\end{theorem}

We also show that the condition cannot be far from sharp in the sense that if the corresponding lower density is strictly greater than $1/2\pi$ then the set cannot be a zero set. Moreover, we also obtain polynomial lower bounds on the derivative of the function $T$ at the zeros which may prove useful when proving interpolation theorems.

Let us also mention the nice paper \cite{MR2944102} which gives sharp interpolation and sampling results for de Branges spaces whose phase function is a doubling measure. In our setting this corresponds to the case $p = 2$ and $U = \overline{\Theta}$ where $\Theta$ is a meromorphic inner function satisfying a global doubling condition. In this case they construct a function similar to ours as well.

\subsection*{Beurling-Malliavin majorants in model spaces}

We now turn to our second application. Let $\omega : \mathbb{R} \to [0,1]$ be a measurable function. The classical Beurling-Malliavin admissible majorant problem is to find a non-trivial function $f \in L^{2}$ with Fourier support in $(-a,a)$ minorizing $\omega$, i.e.
    \begin{equation*}
        \lvert f(x) \rvert \leq \omega(x), \text{ for almost every } x \in \mathbb{R}.
    \end{equation*}
    In this case $\omega$ is called an admissible weight. Experience has shown that this is a very delicate problem. Write $\omega(x) = \exp(-\Omega(x))$. Then a simple necessary condition is the convergence of the logarithmic integral
    \begin{equation*}
        \int \frac{\Omega(x) }{1+x^{2}}dx < \infty.
    \end{equation*}
    However, this condition is by no means sufficient. The well-known criterion of Beurling and Malliavin, from \cite{MR147848}, is that if in addition to the convergence of the logarithmic integral one assumes that $\Omega$ is Lipschitz continuous, then $\omega$ is an admissible weight. This result of Beurling and Malliavin is one of the most celebrated results in the field of uncertainty principles and has found many applications, for example in the fractal uncertainty principle of Bourgain and Dyatlov, \cite{MR3779959} and in the work of Eremenko and Novikov on oscillation of high pass signals, \cite{MR2038065}.

    Since the class of functions with Fourier support in $(-a,a)$ coincides with the model subspace $K_{S_{2a}}$ (up to multiplication by $S_{-a}$) one is naturally led to study the majorant problem in more general model spaces. Let $\Theta$ be an inner function. The model space $K_{\Theta}$ is the orthogonal complement of $\Theta H^{2}$ in $H^{2}$. In terms of Toeplitz kernels we have $K_{\Theta} = N^{2}(\overline{\Theta})$. The systematic study of majorant problems in model spaces was initiated by Havin and Mashreghi in \cite{MR2016246, MR2016247}. This was also taken further by other authors in, for example, \cite{MR2473742, MR2241422, MR2423994}. In fact, as was proven in \cite{MR2016247}, the condition $\norm{\widetilde{\Omega}'}_{\infty} < a/2$ is strong enough to conclude that $\omega$ is an admissible weight. Other results in terms of the Hilbert transform of $\Omega$ were also obtained for certain model spaces. Here we add to that theme.
    \begin{theorem}\label{thm:BM_majorant}
        Let $\Theta$ be a meromorphic inner function and $\omega = \exp(-\Omega)$ a positive function. Suppose
        \begin{enumerate}[(i)]
            \item $\Omega \in L^{1}(\mathbb{P})$, with $d\mathbb{P}(t) = (1+t^{2})^{-1}dt$,
            \item $\widetilde{\Omega} \in C^{1}(\mathbb{R})$,
            \item the function $\alpha = \arg \Theta + 2\widetilde{\Omega}$ is a regular locally doubling weight satisfying
            \begin{equation*}
                1 \lesssim\alpha'(x) \lesssim \langle x\rangle^{K_{0}}, \text{ for some } K_{0} \geq 0.
            \end{equation*}
        \end{enumerate}
        Then $\omega$ is an admissible majorant for $K_{\Theta}$. That is, there exists a non-trivial function $f \in K_{\Theta}$, such that
        \begin{equation*}
            \lvert f(x) \rvert \leq \omega(x), \text{ for all } x \in \mathbb{R}.
        \end{equation*}
    \end{theorem}
    In particular this works whenever the function obeys a power type-law
    \begin{equation*}
        \arg \Theta'(x) + 2\widetilde{\Omega}'(x) \sim \langle x \rangle^{K}, \text{ for some } K \geq 0.
    \end{equation*}

\subsection{Notation}

Before we begin let us fix some notation. We set
\begin{equation*}
    \langle x \rangle = (1+x^{2})^{1/2},
\end{equation*}
and use the asymptotic notation $a \lesssim b$ to mean that there exists a constant $C > 0$, such that
\begin{equation*}
    a(x) \leq Cb(x),
\end{equation*}
where $a$ and $b$ are two nonnegative functions. By $a \sim b$ we mean $a \lesssim b$ and $b\lesssim a$. \newline

The Poisson measure is given by $d\mathbb{P}(t) = (1+t^{2})^{-1}dt$. Let $u \in L^{1}(\mathbb{P})$. \newline
Our real conjugation operator (Hilbert transform) is
\begin{equation*}
    \tilde{u}(x) = \frac{1}{\pi} \text{p.v.}\int \left(\frac{1}{x-t} + \frac{t}{1+t^{2}}\right)u(t)dt.
\end{equation*}
The Poisson integral is
\begin{equation*}
    Pu(z) = \frac{1}{\pi}\int \frac{y}{(x-t)^{2}+y^{2}}u(t)dt.
\end{equation*}
Our complex conjugation operator is
\begin{equation*}
    Qu(z) = \frac{1}{\pi}\int \left(\frac{x-t}{(x-t)^{2}+y^{2}}+\frac{t}{1+t^{2}}\right)u(t)dt.
\end{equation*}
Our Schwartz integral is
\begin{equation*}
    \mathcal{S}u(z) = Pu(z)+iQu(z) = \frac{1}{i\pi} \int \left( \frac{1}{t-z} - \frac{t}{1+t^{2}}\right)u(t)dt.
\end{equation*}

\section{Preliminaries on Toeplitz kernels with real-analytic symbol}

Let $\gamma \in C^{\omega}(\mathbb{R})$ (the space of real analytic functions) and $V(x) = e^{i\gamma(x)}$. Recall that we define the Smirnov Toeplitz kernel as
\begin{equation*}
    N^{+}(V) = \left\{ f \in N^{+} \cap L^{1}_{loc}(\mathbb{R}) : \overline{Vf} \in N^{+} \right\}.
\end{equation*}
First we show that functions in $N^{+}(V)$ are real-analytic.

\begin{lemma}
    Let $\gamma \in C^{\omega}(\mathbb{R})$ and set $V = e^{i\gamma}$. Then $N^{+}(V) \subset C^{\omega}(\mathbb{R})$.
    \begin{proof}
        Let $f \in N^{+}(V)$. Set
        \begin{equation*}
            F(z) = \begin{cases}
                f(z), \; \; z \in \mathbb{C}_{+} \\
                V^{-1}(z)(Vf)(z), \; \;z \in \Omega
            \end{cases},
        \end{equation*}
        where $\Omega \subset \mathbb{C}_{-}$ is a domain where $V$ extends to a zero free analytic function and $(Vf)$ is the extension of the function into the lower half-plane. Then since $f \in L^{1}_{loc}(\mathbb{R})$ the function $F$ extends to an analytic function across the real-line by Morera's theorem, see e.g. Lemma 1.10., on page 133, in \cite{MR2261424}. Since $F(x) = f(x)$ for almost every real $x$ the result follows.
    \end{proof}
\end{lemma}

An important class of unimodular symbols is given by the meromorphic inner functions. When $U = \overline{\Theta}$ the Toeplitz kernel becomes the well-known model space, $N^{p}(U) = K_{\Theta}^{p}$. Choosing $U = S_{-2a}$ we obtain
\begin{equation*}
    S_{a}N^{p}(S_{-2a}) = S_{a}K_{S_{2a}}^{p} = PW_{a}^{p},
\end{equation*}
the Paley-Wiener space consisting of entire functions of exponential type at most $a$ which are $p$-summable on the real-line. We note that the second equality is well-known, but not entirely obvious. Another important class of symbols are ones of the form $U = \overline{\Theta}J$, where $\Theta$ and $J$ are meromorphic inner functions.

Let us record the statement in the introduction about arguments of meromorphic inner functions which are locally doubling.

\begin{lemma}
    Suppose $\alpha$ is the argument of a meromorphic inner function. Then the following are equivalent
    \begin{enumerate}[(i)]
        \item $\alpha'(x)dx$ is a locally doubling measure.
        \item $\alpha'(x) \sim \alpha'(y)$ whenever $d_{\alpha}(x,y) \leq 1$.
        \item $\lvert \alpha''(x) \rvert \lesssim \alpha'(x)^{2}$.
    \end{enumerate}
    \begin{proof}
        A proof is given in \cite{MR2944102}.
    \end{proof}
\end{lemma}

We mention that a theorem of Aleksandrov (Theorem 3.4 in \cite{MR1327512}) shows that in the case of meromorphic inner functions condition $(iii)$ characterizes the so-called one-component inner functions. One-component inner functions are inner functions with the property that there exists $1> \epsilon > 0$, such that the sub level set
\begin{equation*}
    \left\{ z \in \mathbb{C}_{+}: \lvert \Theta(z) \rvert < \epsilon\right\},
\end{equation*}
is connected.

\subsection{Two lemmas needed for the applications}

The next two lemmas are needed for the applications. The result is well-known and goes back to at least Dyakonov \cite{MR1097956}. We denote by $\lvert N^{+}(V)\rvert$ the class of functions on the real line of the form $f = \lvert G \rvert$ with $G \in N^{+}(V)$.

\begin{lemma}\label{lemma:modulus_of_Smirnov_class}
            Let $m \in L^{1}_{loc}(\mathbb{R})$ and $\log m \in L^{1}(\mathbb{P})$. Then $m \in \lvert N^{+}(V) \rvert$ if and only if
            \begin{equation*}
                2 \widetilde{\log m} + \arg(V) = -\arg(I) + 2\pi n,
            \end{equation*}
            for some meromorphic inner function $I$ and measurable integer-valued function $n : \mathbb{R} \to \mathbb{Z}$.
            \begin{proof}
                Suppose first that $m = \lvert G \rvert$ for some $G \in N^{+}(V)$. Without loss of generality we may suppose that $G$ is outer in $\mathbb{C}_{+}$. Moreover, $\overline{VG} \in N^{+}$, and hence there exists an inner function $I$, such that $\overline{VG} = IG$. Since $V$ is real analytic the inner function $I$ is a meromorphic inner function. Passing to the arguments on the real-line and using that $\log m = \log \lvert G \rvert$ we obtain
                \begin{equation*}
                    2 \widetilde{\log m} + \arg(V)= -\arg(I) + 2\pi n,
                \end{equation*}
                where $n$ is some measurable integer-valued function.

                Conversely, if there exists a $I$ and $n$ as in the statement let
                \begin{equation*}
                    F(z) = \exp\left( P[\log m](z) + iQ[\log m](z) \right), \; \; z \in \mathbb{C}_{+}.
                \end{equation*}
                Then $F$ is outer in $\mathbb{C}_{+}$ and $F \in N^{+}$. Moreover by assumption,
                \begin{equation*}
                    VF = e^{i\arg V + i\widetilde{\log m}}\lvert F \rvert = \overline{IF} \in \overline{N^{+}}. 
                \end{equation*}
                Thus $F \in N^{+}(V)$ as claimed.
            \end{proof}    
        \end{lemma}

\begin{lemma}\label{lemma:est_H_transform}
            Let $f \in L^{\infty}(\mathbb{R})$ and suppose $\lvert f'(x) \rvert \lesssim \langle x \rangle^{K}$, for some $K \geq 0$. Then $\widetilde{f}$ is continuous and there exists a constant $C > 0$ (depending on $f$), such that
            \begin{equation*}
                \lvert \widetilde{f}'(x) \rvert \leq C (\log(1+\lvert x\rvert) + \langle x \rangle^{K}).
            \end{equation*}
            \begin{proof}
                We estimate
                \begin{equation*}
                    \left\lvert \int_{1 \geq\lvert x - t \rvert > \delta}\frac{f(t)}{x-t} dt \right\rvert = \left\lvert \int_{1 \geq\lvert s \rvert > \delta}\frac{f(x-s)-f(x)}{s}ds\right\rvert \leq 2C(1+\lvert x \rvert^{K}).
                \end{equation*}
                Now suppose $t > x > 0$, then
                \begin{equation*}
                    \frac{1}{x-t} + \frac{t}{1+t^{2}} = \frac{1+xt}{(x-t)(1+t^{2})} \leq 0.
                \end{equation*}
                Thus for all $x > 0$ we may estimate
                \begin{equation*}
                \begin{split}
                    \left\lvert \int_{x+1}^{\infty} f(t)\left(\frac{1}{x-t} + \frac{t}{1+t^{2}}\right) dt \right\rvert \leq \norm{f}_{\infty} \int_{x+1}^{\infty} -\left(\frac{1}{x-t} + \frac{t}{1+t^{2}}\right)dt \\ = \norm{f}_{\infty} 2^{-1}\log (1+(x+1)^{2}).
                \end{split}
                \end{equation*}
                Similarly if $x > 0$,
                \begin{equation*}
                    \begin{split}
                    \left \lvert \int_{-\infty}^{x-1} f(t) \left(\frac{1}{x-t} + \frac{t}{1+t^{2}}\right)dt \right\rvert \leq \norm{f}_{\infty} \left( -\int_{-\infty}^{-1/x} \frac{1+xt}{(x-t)(1+t^{2})}dt + \int_{-1/x}^{x-1}\frac{1+xt}{(x-t)(1+t^{2})}dt\right).
                    \end{split}
                \end{equation*}
                The first integral on the right hand-side is
                \begin{equation*}
                    -\int_{-\infty}^{-1/x} \frac{1+xt}{(x-t)(1+t^{2})}dt = 2^{-1}\log (1+x^{2}).
                \end{equation*}
                The second integral is
                \begin{equation*}
                    \int_{-1/x}^{x-1} \frac{1+xt}{(x-t)(1+t^{2})}dt = 2^{-1}\log\left( 1+x^{2}\right) + 2^{-1}\log(1+(x-1)^{2}).
                \end{equation*}
                Combining the estimates gives the desired result for all $x > 0$. The case $x < 0$ can be treated similarly.
            \end{proof}
        \end{lemma}
        
\section{Construction of the approximating inner function}

    Let $\alpha$ be a smooth strictly increasing function. Set $d_{\alpha}(x,y) =\lvert \alpha(x)-\alpha(y) \rvert$. We assume that $\alpha'(x)dx$ is a regular locally doubling weight. We remind the reader that this means that $\alpha$ satisfies
    \begin{enumerate}[(i)]
        \item $\alpha'(x) \sim \alpha'(y)$, whenever $d_{\alpha}(x,y) \leq 1$ and $\lvert x \rvert, \lvert y \rvert \geq R$, for some $R \geq 0$.
        \item $\lvert \alpha''(x) \rvert \leq C\alpha'(x)^{2}$, whenever $\lvert x \rvert, \lvert y \rvert \geq R$, for some $R \geq 0$.
    \end{enumerate}
    We shall also need to assume that $\alpha'(x) \gtrsim \langle x \rangle^{\kappa}$, for some $\kappa > -1$. We shall construct a meromorphic inner function $J$ approximating $\alpha$ and with control on $J'$.

    We say that a set is $\alpha$-uniformly separated if there exists $\delta > 0$, such that
    \begin{equation*}
        d_{\alpha}(\lambda,\lambda') > \delta, \text{ for all } \lambda, \lambda' \in \Lambda \text{ with } \lambda \neq \lambda'.
    \end{equation*}
    We say that $\Lambda$ is $\alpha$-uniformly dense if there exists $M > 0$, such that for each $\lambda \in \Lambda$ there exists $\lambda'\neq \lambda$ in $\Lambda$, such that
    \begin{equation*}
        d_{\alpha}(\lambda,\lambda') < M.
    \end{equation*}
    \begin{remark}
        Let us note that the next theorem is much simpler if we do not require any estimates on $J'$. Then we could simply choose
        \begin{equation*}
            \frac{1+J}{1-J} = \int \left( \frac{1}{t-z} - \frac{t}{1+t^{2}}\right)\frac{d\sigma}{i\pi},
        \end{equation*}
        for some discrete measure $\sigma$ supported on $\Lambda$ with $\sigma(\lambda) \lesssim 1$. The difficulty lies in choosing $\sigma$, such that we can control $J'$.
    \end{remark}
    The idea behind the proof of the following theorem is from a lemma appearing in Section 6 of \cite{MR2609247}.
    
    \begin{theorem}\label{thm:construction_inner_function}
        Let $\alpha$ be a regular locally doubling weight, satisfying
        \begin{equation*}
                \langle x\rangle^{-\kappa} \lesssim \alpha'(x), \text{ for some } \kappa < 1.
        \end{equation*}
        Suppose $\Lambda \subset \mathbb{R}$ is an $\alpha$-uniformly discrete with $\alpha$-uniformly bounded gaps. Then there exists a meromorphic inner function $J$, such that
        \begin{equation*}
            \left\{ x \in \mathbb{R} : J(x) = 1 \right\} =\Lambda
        \end{equation*}
        and there exists a constant $N_{0}$, such that
        \begin{equation*}
            \langle \lambda \rangle^{-N_{0}}\alpha'(\lambda)^{-N_{0}} \lesssim \lvert J'(\lambda) \rvert \text{ for all } \lambda \in \Lambda, \text{ and } \lvert J'(x) \rvert \lesssim \langle x \rangle^{N_{0}}\alpha'(x)^{N_{0}}, \text{ for }x \in \mathbb{R}.
        \end{equation*}
        \begin{proof}
            We write $\Lambda = \left\{ \lambda_{n}\right\}_{n\in \mathbb{Z}}$ with $\lambda_{n}$ increasing and $\lambda_{0}$ the smallest positive $\lambda$. Throughout the proof we use the fact that
            \begin{equation*}
                1 \sim d_{\alpha}(\lambda_{n+1}, \lambda_{n}) = \alpha'(t_{n})(\lambda_{n+1}-\lambda_{n}) \sim \alpha'(x)(\lambda_{n+1}-\lambda_{n}), \text{ for all } x \in [\lambda_{n},\lambda_{n+1}].
            \end{equation*}
            In particular we have $\alpha'(\lambda_{n})^{-1} \sim \lambda_{n+1}-\lambda_{n}$ for all $n$.\newline
            
            To construct the inner function $J$ we use a special case of the so-called Krein spectral shift formula from the spectral theory of perturbations of self-adjoint operators, see \cite{MR139006}. Let $\omega_{n}$ be the mid point of the complementary interval $(\lambda_{n},\lambda_{n+1})$. That is $\omega_{n} = 2^{-1}(\lambda_{n+1}+\lambda_{n})$. Define 
            \begin{equation*}
                U = \bigcup_{n \in \mathbb{Z}} (\lambda_{n},\omega_{n}),    
            \end{equation*}
            and set
            \begin{equation*}
                u(x) = \mathbbm{1}_{U}(x) - 1/2.
            \end{equation*}
            Define $J$ by the formula
            \begin{equation}\label{eq:def_J}
                \frac{1+J}{1-J} = \exp \left(\int \left(\frac{1}{t-z}-\frac{t}{1+t^{2}}\right)u(t)dt\right), \; \; z\in \mathbb{C}_{+},
            \end{equation}
            The right hand side maps into the right half-plane since
            \begin{equation*}
                \left\lvert \Im\left( \int \left(\frac{1}{t-z}-\frac{t}{1+t^{2}}\right)u(t)dt\right) \right\rvert = \pi \lvert Pu(z)\rvert \leq \pi/2.
            \end{equation*}
            The function $J$ is thus bounded by $1$ and it is inner since $u = 1/2$ almost everywhere. Moreover, it has a meromorphic extension to the full plane since it extends continuously to $\mathbb{R}$. A straightforward computation shows that $\left\{ J = 1 \right\} = \left\{\lambda_{n}\right\}$ and $\left\{ J = -1 \right\} = \left\{ \omega_{n} \right\}$.
            
            To bound $J'(x)$ on the real line we will use the Aleksandrov-Clark representation of $J$. Let $\sigma_{1}$ be the Aleksandrov-Clark measure of $J$ and $\sigma_{-1}$ the dual measure,
            \begin{equation*}
                \frac{1+J}{1-J} = -ia-ibz + \frac{-i}{ \pi}\int \left( \frac{1}{t-z}-\frac{t}{1+t^{2}}\right)d\sigma_{1}(t),
            \end{equation*}
            where $a \in \mathbb{R}$, $b \geq 0$ and
            \begin{equation*}
                \frac{1-J}{1+J} = -ic-idz + \frac{-i}{ \pi}\int \left( \frac{1}{t-z}-\frac{t}{1+t^{2}}\right)d\sigma_{-1}(t),
            \end{equation*}
            where $c \in \mathbb{R}$, $d \geq 0$. Recall that
            \begin{equation*}
                \sigma_{1} = \sum_{n \in \mathbb{Z}} \sigma_{1}(\left\{ \lambda_{n} \right\})\delta_{\lambda_{n}} \; \text{ and } \; \sigma_{-1} = \sum_{n \in \mathbb{Z}} \sigma_{-1}(\left\{ \omega_{n} \right\})\delta_{\omega_{n}}
            \end{equation*}
            are Poisson-finite. In fact, $b = d = 0$ (but we do not use this fact). Passing to the real line we see that
            \begin{equation*}
                \sigma_{1}(\left\{ \lambda_{n}\right\}) = \frac{2\pi}{\lvert J'(\lambda_{n}) \rvert},
            \end{equation*}
            and
            \begin{equation*}
                \sigma_{-1}(\left\{ \omega_{n}\right\}) = \frac{2\pi}{\lvert J'(\omega_{n}) \rvert}.
            \end{equation*}
            We claim that we have
            \begin{equation}\label{eq:est_below}
                \frac{1}{\langle \lambda_{n} \rangle^{K}\alpha'(\lambda_{n})^{K}}\lesssim \lvert J'(\lambda_{n}) \rvert,
            \end{equation}
            and
            \begin{equation}\label{eq:est_sigmas}
                \sigma_{-1}(\left\{ \omega_{n} \right\}) \lesssim \alpha'(\omega_{n})^{-1}, \text{ and } \lvert J'(\omega_{n})\rvert \lesssim \alpha'(\omega_{n})^{N_{0}} \langle \omega_{n} \rangle^{N_{0}}.
            \end{equation}
            Assuming this for the moment we may complete the proof. We have
            \begin{equation*}
                2\lvert J'(x) \rvert = \lvert1-J(x)\rvert^{2}\lvert \mathcal{S}'\sigma_{1}(x)\rvert \; \text{ and } \; 2\lvert J'(x) \rvert = \lvert1+J(x)\rvert^{2}\lvert\mathcal{S}'\sigma_{-1}(x)\rvert.
            \end{equation*}
            Using the first estimate in the middle third and the second estimate in the left and right thirds of the interval $(\lambda_{n}, \lambda_{n+1})$ gives the upper bound. Indeed, consider $x$ in the left or right third of $(\lambda_{n},\lambda_{n+1})$, then
            \begin{equation*}
                \begin{split}
                    \frac{2\lvert J'(x)\rvert}{\lvert1+J(x)\rvert^{2}} = \frac{1}{\pi}\sum_{k} \frac{\sigma_{-1}(\left\{ \omega_{k} \right\}}{(x-\omega_{k})^{2}} \lesssim \sum_{k\neq n} \frac{\alpha'(\omega_{k})^{-1}}{(x-\omega_{k})^{2}} \\ \sim \left( \int_{\omega_{n+1}}^{\infty}\frac{dt}{(x-t)^{2}} + \int_{-\infty}^{\omega_{n-1}}\frac{dt}{(x-t)^{2}}\right) \sim \lvert \alpha'(\omega_{n}) \rvert.
                \end{split}
            \end{equation*}
            In the middle third we use the representation
            \begin{equation*}
                \frac{2\lvert J'(x) \rvert}{\lvert1-J(x)\rvert^{2}} = \frac{1}{\pi} \sum_{k} \frac{\sigma_{1}(\lambda_{k})}{(x-\lambda_{k})^{2}} \sim \frac{1}{\pi}\sum_{k} \frac{\sigma_{1}(\lambda_{k})}{(\omega_{n}-\lambda_{k})^{2}} = \lvert J'(\omega_{n})\rvert.
            \end{equation*}
            It remains to prove \eqref{eq:est_below} and \eqref{eq:est_sigmas}. We isolate this as a lemma.
            \end{proof}
        \end{theorem}

            \begin{proof}[Proof of equations \eqref{eq:est_below} and \eqref{eq:est_sigmas}]
            We begin with \eqref{eq:est_sigmas}. Note that
            \begin{equation*}
                \frac{1-J}{1+J} = C\exp\left( -Ku(z) \right),
            \end{equation*}
            where $C$ is a constant and $Ku$ is the improper integral
            \begin{equation*}
                Ku(z) = \int\frac{1}{t-z}u(t)dt.
            \end{equation*}
            To compute the derivative of $J$ at $\omega_{n}$ we compute the residue of of $\exp(-Ku(z))$ at $\omega_{n}$
            \begin{equation*}
                 \left\lvert \frac{2}{J'(\omega_{n})} \right\rvert = C\lvert\text{Res}(\exp(-Ku))(\omega_{n})\rvert.
            \end{equation*}          
            The residue of $\exp(-Ku)$ can be estimated as follows. Set
            \begin{equation*}
                A_{n} = \exp\left( -\int_{\mathbb{R}\setminus (\lambda_{n},\lambda_{n+1})} \frac{u(t)}{t-\omega_{n}}\right),
            \end{equation*}
            and
            \begin{equation*}
                \lvert g_{n}(z)\rvert = \left\lvert \exp\left( -\int_{\lambda_{n}}^{\lambda_{n+1}} \frac{u(t)}{t-z}\right)\right\rvert = \frac{\lvert\lambda_{n}-z\rvert^{1/2} \lvert\lambda_{n+1}-z\rvert^{1/2}}{\lvert\omega_{n}-z\rvert}.
            \end{equation*}
            Then
            \begin{equation*}
                \lvert \text{Res}_{\lambda_{n}}(-Ku)(\omega_{n}) \rvert = \lvert A_{n} \rvert \lvert\lambda_{n+1}-\omega_{n}\rvert^{1/2}\lvert \lambda_{n}-\omega_{n}\rvert^{1/2}.
            \end{equation*}
            Since the $\lambda_{n}'s$ are $\alpha$-uniformly dense and $\alpha$-uniformly separated and $\alpha'$ is a regular locally doubling measure
            \begin{equation*}
                \lvert\lambda_{n+1}-\omega_{n}\rvert^{1/2}\lvert \lambda_{n}-\omega_{n}\rvert^{1/2} \sim \frac{1}{\alpha'(\lambda_{n})},
            \end{equation*}
            with constants independent of $n$. We now turn to bounding $A_{n}$. First we claim that $A_{n}$ is bounded above. Indeed, for each term of the form
            \begin{equation*}
                -2\int_{\lambda_{j}}^{\lambda_{j+1}} \frac{u(t)}{t-\omega_{n}} = \log\left( \frac{\lambda_{j}-\omega_{n}}{\omega_{j}-\omega_{n}}\right) + \log\left( \frac{\lambda_{j+1}-\omega_{n}}{\omega_{j}-\omega_{n}}\right) \leq 0,
            \end{equation*}
            since the last inequality is equivalent to
            \begin{equation*}
                \omega_{j}^{2} \geq \lambda_{j+1}\lambda_{j}.    
            \end{equation*}
            Hence, we obtain $A_{n} \leq 1$.

            We now bound $A_{n}$ from below. We assume that $\lambda_{n} > 1$. The negative $\lambda_{n}'s$ are treated similarly. For $j > n$ We compute
            \begin{equation*}
                \begin{split}
                    2\int_{\lambda_{j}}^{\lambda_{j+1}} \frac{u(t)dt}{t-\omega_{n}} = \log(\frac{\omega_{j}-\omega_{n}}{\lambda_{j}-\omega_{n}}) - \log(\frac{\lambda_{j+1}-\omega_{n}}{\omega_{j}-\omega_{n}}) \\= -\frac{\lambda_{j}-\omega_{j}}{\omega_{j}-\omega_{n}} + O\left(\frac{(\lambda_{j}-\omega_{j})^{2}}{(\omega_{j}-\omega_{n})^{2}}\right)- \frac{\lambda_{j+1}-\omega_{j}}{\omega_{j}-\omega_{n}} + O\left(\frac{(\lambda_{j+1}-\omega_{j})^{2}}{(\omega_{j}-\omega_{n})^{2}}\right)\\ = O\left( \frac{\delta_{j}^{2}}{(\omega_{j}-\omega_{n})^{2}}\right),
                \end{split}
            \end{equation*}
            where $\delta_{j} = \omega_{j}-\lambda_{j}$ and we have used the expansion $\log(1+x) = x +O(x^{2})$. Thus
            \begin{equation*}
                \begin{split}
                    2\int_{\lambda_{n+1}}^{\infty} \frac{u(t)dt}{t-\omega_{n}} = \sum_{j > n} \int_{\lambda_{j}}^{\lambda_{j+1}}\frac{u(t)dt}{t-\omega_{n}} \lesssim \sum_{j > n} \frac{\delta_{j}^{2}}{(\omega_{j}-\omega_{n})^{2}}.
                \end{split}
            \end{equation*}
            Using that $\alpha'$ is locally doubling we bound the sum as follows 
            \begin{equation*}
                \sum_{j > n} \frac{\delta_{j}^{2}}{(\omega_{j}-\omega_{n})^{2}} = \frac{1}{2}\sum_{j > n} \int_{\lambda_{j}}^{\lambda_{j+1}} \frac{\delta_{j}}{(\omega_{j}-\omega_{n})^{2}} \lesssim \int_{\lambda_{n+1}}^{\infty} \frac{dt}{\alpha'(t)(t-\omega_{n})^{2}}.
            \end{equation*}
            We split the last integral into two parts
            \begin{equation*}
                \int_{\lambda_{n+1}}^{\infty} \frac{dt}{\alpha'(t)(t-\omega_{n})^{2}} = \int_{\lambda_{n+1}}^{\lambda_{n}+\omega_{n}} \frac{dt}{\alpha'(t)(t-\omega_{n})^{2}} + \int_{\lambda_{n}+\omega_{n}}^{\infty} \frac{dt}{\alpha'(t)(t-\omega_{n})^{2}} = I_{1} + I_{2}.
            \end{equation*}
            The first integral we bound using integration by parts, $\lvert \alpha''(x) \rvert \lesssim \alpha'(x)^{2}$
            \begin{equation*}
                \begin{split}
                    I_{1} = \left(\frac{1}{\alpha'(\lambda_{n+1})(\lambda_{n+1}-\omega_{n})}-\frac{1}{\alpha'(\lambda_{n}+\omega_{n})\lambda_{n}}\right) - \int_{\lambda_{n+1}}^{\omega_{n}+\lambda_{n}} \frac{\alpha''(t)}{\alpha'(t)^{2}}\frac{1}{t-\omega_{n}}dt \\ \lesssim \log(\lambda_{n})+ \lvert\log(2\delta_{n})\rvert.
                \end{split}
            \end{equation*}
            The second integral $I_{2}$ is simpler to bound (we only use that $\alpha' \gtrsim \langle x\rangle^{-\kappa}$)
            \begin{equation*}
                I_{2} \leq \int_{\omega_{n}+\lambda_{n}}^{\infty}\frac{t^{\kappa}}{(t-\omega_{n})^{2}} \lesssim \frac{1}{\lambda_{n+1}^{1-\kappa}}.
            \end{equation*}
            The case when $j < n$ is similar. All in all we obtain a bound of the form
            \begin{equation*}
                \langle \delta_{n} \lambda_{n} \rangle^{-K_{0}}\lesssim A_{n} \lesssim 1.
            \end{equation*}
            Thus we obtain
            \begin{equation*}
                \alpha'(\omega_{n}) \lesssim \lvert J'(\omega_{n}) \rvert \lesssim \langle \omega_{n} \rangle^{N_{0}} \alpha'(\omega_{n})^{N_{0}}.
            \end{equation*}
            Which is equation \eqref{eq:est_sigmas}. To obtain the lower bound for $J'(\lambda_{n})$ we instead consider
            \begin{equation*}
                \frac{1+J}{1-J} = C \exp(Ku).
            \end{equation*}
            The only difference is that the lower bound is the worse bound since for the corresponding $A_{n}$ we have $A_{n} \geq 1$. Exactly the same argument shows that
            \begin{equation*}
                \langle\lambda_{n} \rangle^{-N} \alpha'(\lambda_{n})^{-N} \lesssim \lvert J'(\lambda_{n})\rvert,
            \end{equation*}
            which is \eqref{eq:est_below}.
    \end{proof}

    Theorem \ref{thm:main} now follows as a corollary.

    \begin{proof}[Proof of Theorem \ref{thm:main}]
            The result follows from the previous theorem after setting
            \begin{equation*}
                \Lambda = \left\{ \lambda \in \mathbb{R} : \alpha(\lambda) = 0 \mod 2\pi \right\}.
            \end{equation*}
    \end{proof}
    
    \section{Zero sets in Toeplitz kernels}\label{sec:zero_sets}

        We shall use the following characterization of functions absolute values of functions in Toeplitz kernels. In this section all symbols are real analytic and unimodular.

        Applying Lemma \ref{lemma:modulus_of_Smirnov_class} to $\arg(V)= -\alpha+h + \arg(J) + 2\epsilon x$ we obtain that $m$ is the absolute value of a $N^{+}(JUS_{2\epsilon})$ function if and only if
        \begin{equation*}\label{eq:to_be_solved}
            2\widetilde{\log m} = \alpha - h- 2\epsilon x-\arg(J) - \arg(I) + 2\pi n.
        \end{equation*}
        Assuming that $m$ is smooth the integer valued function $n$ is constant and hence may be baked in to the argument of the inner function $I$. The function $m$ will be constructed by obtaining sufficiently good bounds on the Hilbert transform of the right hand-side.

        \subsection{A preparatory lemma}
        
    In order to prove our theorem we shall need to regularize the set $\Lambda$ to not contain large gaps.
    
    Recall the definition of $\alpha$-Beurling density as
    \begin{equation*}
        D^{+}_{\alpha}(\Lambda) = \lim_{r \to \infty} \sup_{\alpha(I) = r} \frac{\#(\Lambda \cap I)}{r},
    \end{equation*}
    and
    \begin{equation*}
        \alpha(I) = \int_{I} \alpha'(x)dx.
    \end{equation*}
    For a discrete set $\Lambda \subset \mathbb{R}$ we denote by $n_{\Lambda}$ the counting function of $\Lambda$, i.e. the piece-wise constant function which is zero at $0$ and increases by one at each $\lambda \in \Lambda$.
    \begin{lemma}\label{lemma:regularization_of_lambda}
        Let $\alpha$ be a regular locally doubling measure and $\Lambda \subset \mathbb{R}$ be an $\alpha$-uniformly discrete set. Suppose $D^{+}_{\alpha}(\Lambda) < 1/2\pi$. Then for sufficiently small $\delta > 0$ there exists an $\alpha$-uniformly discrete and $\alpha$-uniformly dense set $\Lambda \subset \Lambda'$ such that $(1-\delta)\alpha-2\pi n_{\Lambda'}$ is bounded.
        \begin{proof}
            Fix a positive integer $N_{0} > 0$ and $\delta > 0$, such that
            \begin{equation*}
                \#(\Lambda \cap I) < (1-\delta)\frac{\alpha(I)}{2\pi},
            \end{equation*}
            for all intervals $I$ of $\alpha$-length $\alpha(I)  \geq N_{0}$. Now we decompose the real line into disjoint intervals of $\alpha$-length $N_{0}$, $I_{j} = [a_{j}, a_{j+1})$, $j \in \mathbb{Z}$. That is $d_{\alpha}(a_{j+1},a_{j}) = N_{0}$. Now we add points to each interval keeping the set $\alpha$-uniformly separated and such that
            \begin{equation*}
                \#(\Lambda \cap I_{j}) = \left[ (1-\delta)\frac{\alpha(I_{j})}{2\pi} \right].
            \end{equation*}
            Now we may inductively add one extra point (keeping the uniform separation) to the intervals $I_{j}$ if necessary to ensure that
            \begin{equation*}
                (1-\delta)\alpha(a_{j}) \geq 2\pi n_{\Lambda'}(a_{j}) \geq (1-\delta)\alpha(a_{j}) - 1.
            \end{equation*}
            By construction $\Lambda'$ is $\alpha$-uniformly separated and $\alpha$-uniformly dense. The previous inequality shows that $(1-\delta)\alpha-2\pi n_{\Lambda'}$ is bounded.
        \end{proof}
    \end{lemma}

    \subsection{Construction of a function with prescribed zeros}
    
    We are now ready to construct the vanishing function when the phase function is a locally doubling weight.
        \begin{proof}[Proof of Theorem \ref{thm:zero_function}]
            Using Lemma \ref{lemma:regularization_of_lambda} for $\delta > 0$ sufficiently small we construct an an $\alpha$-discrete set $\Lambda'$ containing $\Lambda$, such that
            \begin{equation*}
                (1-\delta)\alpha -2\pi n_{\Lambda'}, \; \text{ is bounded},   
            \end{equation*}
            and $\Lambda' \setminus \Lambda$ is infinite. Using Theorem \ref{thm:construction_inner_function} we construct a meromorphic inner function $J$, such that
            \begin{equation*}
                \left\{ x \in \mathbb{R} : J(x) = 1 \right\} = \Lambda',
            \end{equation*}
            and
            \begin{equation*}
            \lvert J'(x) \rvert \lesssim \langle x\rangle^{K_{0}}, \text{ and } \langle \lambda \rangle^{-K_{0}} \lesssim\lvert J'(\lambda) \rvert, \text{ for all } \lambda \in \Lambda'.
            \end{equation*}
            Then $\arg(J)-2\pi n_{\Lambda'}$ is bounded.

            Let $m$ be a locally summable function, such that $\log(m) \in L^{1}(\mathbb{P})$. As we have seen $m$ is equal to the absolute value of a $N^{+}(JUS_{2\epsilon})$ function if there exists an inner function $I$, such that
            \begin{equation*}
                2\widetilde{\log m} = \alpha + h -\arg(J) - 2\epsilon x - \arg(I).
            \end{equation*}
            By construction $(1-\delta)\alpha - \arg(J)$ is a bounded function. For sufficiently small $\epsilon$ we have $\delta \alpha' - 2\epsilon > 0$. Thus it follows from Theorem \ref{thm:construction_inner_function} that there exists a meromorphic inner function $I$, such that $\delta \alpha -2\epsilon x - \arg(I)$ is bounded and $\arg(I)' \lesssim \langle x \rangle^{K_{0}}$. Define $m$ by
            \begin{equation*}
                - 2\log m = \tilde{h} + \widetilde{\left((1-\delta)\alpha - \arg(J)\right)} + \widetilde{\left(\delta \alpha-2\epsilon x - \arg(I)\right)}.
            \end{equation*}
            Thus there exists $G \in N^{+}(JUS_{2\epsilon})$, such that
            \begin{equation*}
                \lvert G \rvert = m.
            \end{equation*}
            It remains to prove the desired estimates. The function
            \begin{equation*}
                u = \frac{\alpha + h -\arg(J) - 2\epsilon x - \arg(I)}{2},
            \end{equation*}
            is bounded and satisfies $\lvert u'(x) \rvert\leq C\langle x \rangle^{N}$ for some $C$ and $N$, Hence by Lemma \ref{lemma:est_H_transform} its Hilbert transform satisfies $\lvert \widetilde{u}'(x) \rvert \lesssim \langle x \rangle^{N}$, $N \geq 1$. Set $M = \norm{u}_{\infty}$. Then
            \begin{equation*}
                m(x) = \left( e^{\widetilde{u/M}}\right)^{M}.
            \end{equation*}
            We shall estimate $q = \exp(\widetilde{u/M})$. We write
            \begin{equation*}
                \lvert q(x) \rvert \leq \lvert q(0) \rvert + \int_{0}^{x} \frac{1}{M}\lvert \widetilde{u}'(t) \rvert e^{\widetilde{u/M}}dt \leq \lvert q(0) \rvert + \frac{C}{M}\langle x \rangle^{N+2}\int_{-\infty}^{\infty}e^{\lvert \widetilde{u/M} \rvert}d\mathbb{P}(t).
            \end{equation*}
            The last integral is finite since $u/M$ is bounded by $1$ (the Zygmund estimate, see Corollary 2.6. in Chapter 3 of \cite{MR2261424}). Similarly,
            \begin{equation*}
                \begin{split}
                    \lvert q(x)^{-1} \rvert = e^{-\widetilde{u/M}} \leq \lvert q(0)^{-1} \rvert + \int_{0}^{x} \frac{1}{M} \lvert \widetilde{u}'(t) \rvert e^{-\widetilde{u/M}}dt \\ \leq \lvert q(0)^{-1} \rvert + \frac{C}{M}\langle x \rangle^{N+2}\int_{-\infty}^{\infty}e^{\lvert \widetilde{u/M} \rvert}d\mathbb{P}(t).         
                \end{split}
            \end{equation*}
            Thus we have
            \begin{equation*}
                \langle x \rangle^{-N} \lesssim m(x) \lesssim \langle x \rangle^{N}.
            \end{equation*}
            It remains to fix the polynomial growth at infinity. This can be done by dividing out zeros of the function $1-J$ at points in $\Lambda' \setminus \Lambda$. The function
            \begin{equation*}
                T = c(1-J)G
            \end{equation*}
            satisfies all the requirements.
        \end{proof}

    \subsection{Necessity of the density condition}

    For our necessary condition we shall need the notion of uniform lower density.
    \begin{equation*}
        D_{\alpha}^{-}(\Lambda) = \lim_{r \to \infty}\inf_{\alpha(I) = r} \frac{\#(\Lambda \cap I)}{r},
    \end{equation*}
    where the infimum is taken over all intervals $I$ of $\alpha$-length $r$. Here we will give a simple argument that the condition $D^{+}(\Lambda) < 1/2 \pi$ is not too far from optimal.
    \begin{theorem}
        Let $U = \exp(i(-\alpha+h))$ be unimodular symbol with $\alpha,  h\in C^{\omega}$. $\alpha'  \gtrsim 1$ and $h \in L^{\infty}$. Suppose $D^{-}_{\alpha}(\Lambda) > 1/2\pi$. Then $\Lambda$ is not a zero set of $N^{+}(U)$.
        \begin{proof}
            Let $\delta > 0$ and $N_{0} > $ be such that for all intervals of length $\alpha(I) \geq N_{0}$ we have
            \begin{equation*}
                2\pi\# \left( \Lambda \cap I \right) > (1+\delta)\alpha(I),
            \end{equation*}
            for all interval $I$, such that $\alpha(I) \geq N_{0}$. Let $\left\{a_{j} \right\}_{j}$ be an increasing sequence, such that $d_{\alpha}(a_{j+1},a_{j}) = N_{0}$ and $a_{0} = 0$. Then
            \begin{equation*}
                2\pi \#\left( \Lambda \cap (a_{j},a_{j+1}]\right) \geq (1+\delta)\alpha(I_{j}).
            \end{equation*}
            Summing from $j = 0$ to $j = n-1$ gives
            \begin{equation*}
                2\pi n_{\Lambda}(a_{n}) - \alpha(a_{n}) > \delta\alpha(a_{n}) \geq \delta\epsilon a_{n},
            \end{equation*}
            for sufficiently small $\epsilon > 0$. Thus the left hand-side is not the sum of a Hilbert transform of an $L^{1}(\mathbb{P})$ function and a decreasing function. The result follows from the next lemma. Very similar results appear in \cite{MR2215727}.
        \end{proof}
    \end{theorem}
    \begin{lemma}
        Let $\arg U=-\alpha + h$ with $\alpha, h \in C^{\omega}$, $\alpha' >0$ and $h$ bounded. Then $\Lambda$ is a zero set of $N^{+}(U)$ if and only if
        \begin{equation*}
            2\pi n_{\Lambda} - \alpha = -\gamma + \widetilde{u},
        \end{equation*}
        where $u \in L^{1}(\mathbb{P})$ and $\gamma$ is the argument of a meromorphic inner function.
        \begin{proof}
            First we claim that $N^{+}(U)$ contains a non-trivial function vanishing on $\Lambda \subset \mathbb{R}$ if and only if for one/any meromorphic inner function $J$ with $\Lambda = \left\{ x \in \mathbb{R} : J(x) = 1\right\}$ the Toeplitz kernels $N^{+}(JU)$ is non-trivial. Indeed if $f \in N^{+}(U)$ vanishes on $\Lambda$ then $f/(1-J) \in N^{+}(UJ)$ and conversely, if $g \in N^{+}(UJ)$ then $(1-J)g \in N^{+}(U)$ and the function vanishes on $\Lambda$.
            
            Now we use the fact that $N^{+}(UJ) \neq \left\{ 0 \right\}$ if and only if there exists a meromorphic inner function $I$ and a real-analytic outer function $H$ with no real zeros, such that
            \begin{equation*}
                JU = \overline{I}\frac{\overline{H}}{H},
            \end{equation*}
            see for example the first line in the proof of Proposition 3.13. in \cite{MR2215727}. Passing to arguments on the real-line we obtain
            \begin{equation*}
                \arg(J) + \arg U = -\arg(I) - 2\widetilde{\log\lvert H \rvert} + C,
            \end{equation*}
            since $H$ has no zeros. Since $2\pi n_{\Lambda} - J$ is bounded we have
            \begin{equation*}
                2\pi n_{\Lambda}-\alpha = (-\arg(I)+C) - 2\widetilde{\log\lvert H \rvert} - h - (2\pi n_{\Lambda}-\arg(J)).
            \end{equation*}
            The function $u = - 2\widetilde{\log\lvert H \rvert} - h - (2\pi n_{\Lambda}-\arg(J))$ is the conjugate of a $L^{1}(\mathbb{P})$ function. Hence the result follows.
        \end{proof}
    \end{lemma}

\section{Beurling-Malliavin majorants for model spaces}\label{sec:bm_majorants}

    We shall use the representation for the absolute values of elements of a Toeplitz kernel from Lemma \ref{lemma:modulus_of_Smirnov_class}. Indeed, if we can find a measurable and nonnegative function $m$ with $\log m \in L^{1}(\mathbb{P})$ such that $m \omega = \lvert f\rvert$ with $f \in N^{+}(\overline{\Theta})$, and such that $m$ is bounded then
    \begin{equation*}
        \lvert f(x) \rvert = m(x)\omega(x) \leq \norm{m}_{\infty}\omega(x),
    \end{equation*}
    and we would be done. It turns out to be easiest to first construct $m$ which may grow as a polynomial and then correct afterwards.

    \begin{proof}[Proof of Theorem \ref{thm:BM_majorant}]
        Let $m \geq 0$ and suppose $m \in L^{1}(\mathbb{P})$. Then by Lemma \ref{lemma:modulus_of_Smirnov_class} $m\omega$ is the absolute value of a $N^{+}(S_{2\epsilon}\overline{\Theta})$ function if there exists an inner function $I$, such that
        \begin{equation*}
            2\widetilde{\log(m)} = 2\widetilde{\Omega} + \arg \Theta -2\epsilon x- \arg I.
        \end{equation*}
        Let $\epsilon > 0$ be so small that $\alpha = 2\widetilde{\Omega} + \arg \Theta$ satisfies $\alpha'(x) - 2\epsilon \gtrsim 1$. By assumption the function $\alpha - 2\epsilon x$ satisfies the assumptions of Theorem \ref{thm:main} (for sufficiently small $\epsilon > 0$ the perturbation does not destroy the local doubling condition) and hence there exists a meromorphic inner function $I$, such that $\alpha-2\epsilon x- \arg(I) \in L^{\infty}$ and
        \begin{equation*}
            \lvert \alpha'(x)-2\epsilon - \arg(I)'(x) \rvert \lesssim \langle x \rangle^{K}, \text{ for some } K \in \mathbb{N}.
        \end{equation*}
        Thus defining $m$ via
        \begin{equation*}
            2\log(m) = -\widetilde{\left( 2\widetilde{\Omega} + \arg \Theta-2\epsilon x - \arg I\right)},
        \end{equation*}
        there exists $f \in N^{+}(S_{2\epsilon}\overline{\Theta})$, such that $m = \lvert f \rvert$. it remains to prove the estimates. Set
        \begin{equation*}
            u(x) = -\widetilde{\Omega}(x) - 2^{-1}\arg \Theta(x)-\epsilon x + 2^{-1}\arg I(x).
        \end{equation*}
        Then $m(x) = \exp(\widetilde{u})$. The function $u$ is bounded and satisfies $\lvert u'(x) \lvert \lesssim \langle x \rangle^{K}$, $K \geq 1$. Hence, by Lemma \ref{lemma:est_H_transform}, its Hilbert transform satisfies $\lvert \widetilde{u}'(x) \rvert \lesssim \langle x \rangle^{K}$. Let $M = \norm{u}_{\infty}$. Then $v = u/M$ is bounded by one and hence by the Zygmund estimate (Corollary 2.6. in Chapter 3 of \cite{MR2261424})
        \begin{equation*}
            \int_{-\infty}^{\infty} e^{\lvert \widetilde{v}(x)\rvert}d\mathbb{P}(x) <\infty.
        \end{equation*}
        We may bound the function $h = m^{1/M} = \exp(\widetilde{u}/M)$ as follows
        \begin{equation*}
            \lvert h(x) \rvert \leq \lvert h(0) \rvert + \frac{1}{M}\int^{x}_{0} \lvert \widetilde{u}'(x) \rvert h(x)dx \lesssim 1 + \langle x \rangle^{N+2}\int_{-\infty}^{\infty} e^{\lvert \widetilde{u(x)/M}\rvert}d\mathbb{P}(x) \lesssim\langle x \rangle^{N+2}.
        \end{equation*}
        Hence, we obtain
        \begin{equation*}
            \lvert f(x) \rvert = m(x)\omega(x)\lesssim \omega(x)\langle x \rangle^{N_{0}},
        \end{equation*}
        for some non-trivial $f \in K^{+}_{\Theta}$. It remains to fix the polynomial growth at infinity. For any $N_{0} > 0$ we may choose $\delta > 0$ so small that $N\delta < \epsilon$. Then the function
            \begin{equation*}
                h(z) = \left(\frac{\sin(z\delta)}{z\delta}\right)^{N_{0}+1},
            \end{equation*}
            is an entire function of exponential type less than $\epsilon$ satisfying
            \begin{equation*}
                 \lvert h(x) \rvert \leq \langle x \rangle^{-(N_{0}+1)}.
            \end{equation*}
            The function $S_{\epsilon}h$ belongs to $N^{+}(S_{-2\epsilon})$ and hence $hf \in K_{\Theta}$, moreover,
            \begin{equation*}
                \lvert h(x)f(x)\rvert \lesssim\omega(x), \text{ for all } x \in \mathbb{R}.
            \end{equation*}
    \end{proof}
    
\bibliographystyle{abbrv}
\bibliography{citations}

\end{document}